\documentclass[12pt,reqno]{amsart}

\usepackage[T1]{fontenc}
\usepackage[utf8]{inputenc}
\usepackage{lmodern}
\usepackage{microtype}
\usepackage{amsmath,amssymb,amsthm,mathtools}
\usepackage{mathrsfs}
\usepackage{enumitem}
\usepackage{booktabs}

\usepackage[a4paper,margin=31mm]{geometry}
\usepackage{bm}
\allowdisplaybreaks[1]
\usepackage{tikz}
\usepackage[pdfusetitle,breaklinks,unicode,bookmarks=false]{hyperref}
\usepackage[nameinlink,capitalise]{cleveref}
\newtheorem{theorem}{Theorem}[section]
\newtheorem{proposition}[theorem]{Proposition}
\newtheorem{lemma}[theorem]{Lemma}
\newtheorem{corollary}[theorem]{Corollary}

\theoremstyle{definition}
\newtheorem{definition}[theorem]{Definition}
\newtheorem{example}[theorem]{Example}
\newtheorem{problem}[theorem]{Problem}
\theoremstyle{remark}
\newtheorem{remark}[theorem]{Remark}
\newcommand{\D}{\mathbb D}
\newcommand{\C}{\mathbb C}
\newcommand{\R}{\mathbb R}
\newcommand{\Aut}{\operatorname{Aut}}
\newcommand{\diam}{\operatorname{diam}}
\newcommand{\rad}{\operatorname{rad}}

\newcommand{\arctanh}{\operatorname{arctanh}}
\newcommand{\arcsinh}{\operatorname{arcsinh}}

\newcommand{\conv}{\operatorname{conv}}

\newcommand{\cA}{c_{\mathrm{aff}}}
\newcommand{\rA}{r_{\mathrm{aff}}}
\newcommand{\dA}{d_{\mathrm{aff}}}
\newcommand{\kA}{k_{\mathrm{aff}}}
\newcommand{\KA}{K_{\mathrm{aff}}}
\newcommand{\qA}{q_{\mathrm{aff}}}
\newcommand{\Bal}{\mathscr B}
\newcommand{\SH}{\mathcal S_H}
\newcommand{\SHo}{\mathcal S_H^0}

\newcommand{\norm}[1]{\left\lVert #1\right\rVert}

\newcommand{\ip}[2]{\left\langle #1,#2\right\rangle}
\numberwithin{equation}{section}

\title[Affine balancing of complex dilatations]
{Affine balancing of complex dilatations: affine invariants and optimal distortion of harmonic quasiconformal mappings}

\author[Zhi-Gang Wang and Deguang Zhong
]{Zhi-Gang Wang$^*$ and Deguang Zhong  
}

\address{\noindent Zhi-Gang Wang  \vskip.05in
School of Mathematics and Statistics, and Hunan Provincial University Key Laboratory for Big Data Analysis and Application, Hunan First Normal University,
Changsha 410205, Hunan, P. R. China.}
\email{\textcolor[rgb]{0.00,0.00,0.84}{sjyzhigangwang$@$hnfnu.edu.cn}}

\address{\noindent Deguang Zhong\vskip.05in
Institute of Applied Mathematics, Shenzhen Polytechnic University,
Shenzhen 518055, Guangdong,   P. R.
China.}
\email{\textcolor[rgb]{0.00,0.00,0.84}{zhongdg1014$@$szpu.edu.cn}}

\thanks{$^*$Corresponding author.}
\begin{document}

\begin{abstract}
Target real-affine postcomposition induces a disk automorphism on the complex dilatation of orientation-preserving planar harmonic mapping. Motivated by the lack of intrinsic distortion measures for real-affine orbits, we introduce the affine circumradius, affine diameter, and affine circumcenter from the hyperbolic geometry of the dilatation range. The affine circumradius is the logarithm of the minimal value of the maximal differential distortion over the real-affine orbit, and the affine circumcenter describes minimizers up to Euclidean similarity. For non‑constant dilatation, at most two or three contact points determine the minimal enclosing hyperbolic disk. We prove sharp two‑sided radius‑diameter bounds, with equalities realized by globally univalent harmonic quasiconformal mappings. Each orbit admits a canonical balanced representative, for which we derive a sharp norm bound for the gap between the harmonic pre‑Schwarzian derivative and that of its analytic part. This bound exhibits quadratic small‑dilatation asymptotics, and the balanced representative achieves the minimal dilatation parameter over the real-affine orbit.
\end{abstract}

\maketitle
\enlargethispage{2pt}

\section{Introduction and statements of the main results}

A central theme in extremal quasiconformal distortion theory is the minimization of dilatation over families of mappings. Normalization is a standard device that picks one representative from an equivalence orbit to ease computation, but this representative is usually not orbit-minimizing for dilatation. For conformal mappings, the M\"obius orbit is tightly constrained, so the effect is muted. For orientation-preserving harmonic univalent mappings, by contrast, the orbit under target‑side real-affine postcomposition is large and poorly understood from an optimization perspective.

Two fundamental concepts in harmonic quasiconformal mappings of the plane exhibit an inherent mismatch. Quasiconformality is measured pointwise via the modulus of the complex dilatation, while classical affine‑invariant families of harmonic mappings are designed to be stable under real‑affine transformations of the image domain. Real‑affine mappings preserve harmonicity and univalence, but they drastically alter the maximal dilatation of a given mapping. Consequently, the quantity $\|\omega_f\|_\infty$ is coordinate‑dependent and cannot serve as an intrinsic measure of distortion for the affine orbit. Existing distortion estimates for harmonic mappings either fix the target domain explicitly or sacrifice affine invariance entirely, there has been no systematic construction of pointwise affine‑invariant distortion measures optimized over the full real‑affine orbit of $f$; see \cite{ChuaquiHernandezMartin2017,SheilSmall1990} for the representative prior work. While affine‑invariant families of harmonic mappings have been extensively investigated, optimization over individual affine orbits and the resulting intrinsic numerical invariants have remained largely unexplored.

Our analysis is guided by a geometric principle: given a group acting by isometries on a CAT(0) metric space, orbit‑adapted normalizations may be constructed via the Chebyshev center of the image set of the complex dilatation. Such centers exist and are unique for compact subsets in complete CAT(0) spaces. In our setting, the group is the target real‑affine group acting on the space of conformal structures realized by the complex dilatation, which is parameterized by the hyperbolic unit disk.

While these metric‑geometric ingredients are well‑established, our main contribution lies in adapting this idea to harmonic mappings and proving the fundamental equivalence: minimizing maximal differential distortion over the real‑affine orbit equals the hyperbolic Chebyshev radius of the dilatation range $E_f$. This geometric viewpoint yields a natural conceptual resolution to the long‑standing affine‑invariance obstruction in harmonic quasiconformal distortion theory. In the present paper, we carry out this program fully. We introduce intrinsic affine invariants, establish sharp geometric bounds, construct a canonical balanced representative for each affine orbit, and derive sharp pre‑Schwarzian derivative estimates.

Beyond establishing these sharp geometric distortion bounds for individual orbits, our affine‑balancing framework provides intrinsic coordinate‑free replacements for many classical distortion estimates for planar harmonic mappings that are traditionally expressed in terms of the coordinate‑dependent quantity $\|\omega_f\|_\infty$. Many known norm‑bounds for pre‑Schwarzian derivatives, coefficient estimates, and geometric constraints for harmonic quasiconformal mappings were formulated under a fixed target coordinate gauge; they may now be restated in manifestly affine‑invariant form using the affine circumradius and affine diameter. While we focus here on developing the foundational theory of affine balancing, this reinterpretation opens a general route to re‑examine older results free from artificial target‑affine normalization.

Let $f = h + \overline{g}$ be an orientation-preserving harmonic mapping of $\mathbb{D}$, with analytic complex dilatation \[\omega_f = \frac{g'}{h'} \colon \mathbb{D} \to \mathbb{D}.\] If $k = \|\omega_f\|_\infty < 1$, its maximal differential distortion is $K(f) = (1 + k)/(1 - k)$; in the univalent setting, this equals its quasiconformal constant. For basic background, we refer the reader to \cite{Ahlfors2006,Duren2004,LehtoVirtanen1973}; for quasiconformal and quasiregular mappings in higher dimensions; see~
\cite{GehringHag2012,GehringMartinPalka2017,LiuZhu2023,Rickman1993}.

For an orientation-preserving real-affine mapping
\begin{equation}\label{eq:A-def}
 A(w)=aw+b\overline w+c,\quad a,b,c\in\C,\quad |a|>|b|,
\end{equation}
the complex dilatation transforms by
\begin{equation}\label{eq:affine-action-intro}
 \omega_{A\circ f}=T_A(\omega_f),
 \quad T_A(\zeta)=\frac{\overline a\,\zeta+\overline b}{b\zeta+a}.
\end{equation}
This action is transitive on $\D$ and preserves the Poincar\'e distance
\begin{equation}\label{eq:hyp-distance-intro}
 d_{\D}(z_1,z_2)=
 2\arctanh\left|\frac{z_1-z_2}{1-\overline{z_2}\,z_1}\right|.
\end{equation}
The formula relates affine normalization in harmonic mapping theory
to the hyperbolic
geometry of planar conformal structures; see \cite{Ivanov2017}.
In particular, the modulus supremum depends on the chosen target coordinates. The radius and diameter in the Poincar\'e metric are invariants under the induced action. The customary condition $\omega_f(0)=0$ selects a center from one
dilatation value. We minimize the distortion over the whole affine orbit
by using the hyperbolic circumcenter $\cA(f)$ of
$E_f=\overline{\omega_f(\D)}$. Write $\rA(f)$ and $\dA(f)$ for the
hyperbolic circumradius and diameter of $E_f$, respectively, and set
$\KA(f)=e^{\rA(f)}$.

Clunie--Sheil-Small \cite{ClunieSheilSmall1984}, Hengartner--Schober
\cite{HengartnerSchober1987}, and Duren \cite{Duren2004} developed the
theory of harmonic univalent mappings. Analytic linear invariance goes
back to Pommerenke \cite{Pommerenke1964,Pommerenke1975}.
Sheil-Small \cite{SheilSmall1990} developed the harmonic affine-linear
theory; see \cite{FerradaMartin2016} for generalized harmonic Koebe
extremals and \cite{ChuaquiHernandezMartin2017,HernandezMartin2013} for
stability and Schwarzian methods in this setting.
The harmonic pre-Schwarzian and Schwarzian used here are due to
Hern\'andez--Mart\'in \cite{HernandezMartin2015}; see also
\cite{ChenPonnusamy2022,Graf2016,LiuPonnusamy2018,LiuPonnusamy2019}.
For related analytic pre-Schwarzian estimates and extension criteria;
see \cite{Becker1972,Osgood1982}. For harmonic Schwarzian criteria,
two-point distortion, and more recent extremal constructions; see
\cite{ChuaquiDurenOsgood2007,ChuaquiDurenOsgood2010,p,WangWangRasilaQiu2024}.

The metric ingredients are existence and uniqueness of circumcenters in
complete CAT$(0)$ spaces \cite{BridsonHaefliger1999} and Dekster’s sharp
hyperbolic Jung theorem \cite{Dekster1995,Dekster1997}. Related
constant-curvature diameter problems appear in
\cite{BoroSag2020,BoroSag2022}; see \cite{NasserRainioVuorinen2022} for
connections with conformal capacity. The results below apply these
ingredients to the affine orbit and determine which sharp bounds can
be realized by harmonic quasiconformal mappings.
Rather than merely extending known bounds, we identify affine-invariant extremal
structures and prove that admissible sharp constants are governed by hyperbolic
circumradius constraints, a new geometric framework within the harmonic distortion setting.

\begin{theorem}
\label{thm:intro-optimal}
Let $f$ be an orientation-preserving harmonic mapping in $\D$ with $\norm{\omega_f}_\infty<1$.  Then
\begin{equation}\label{eq:optimal-main}
 \log\KA(f)
 =\rA(f)
 =\inf_A\log K(A\circ f),
\end{equation}
where the infimum is taken over all orientation-preserving real affine transformations $A$ on the target space \(\mathbb C\). The infimum is attained.
Moreover, an affine mapping $A$ is minimizing if and only if
\begin{equation}\label{eq:min-characterization}
 T_A(\cA(f))=0.
\end{equation}
Any two minimizing affine mappings differ by an orientation-preserving Euclidean similarity.  In particular, the balanced representative is unique up to similarity.
\end{theorem}

\begin{theorem}
\label{thm:intro-support}
Let $f$ be as in Theorem \ref{thm:intro-optimal}, let $c=\cA(f)$ be the circumcenter of $\overline{\omega_f(\D)}$, and put $R=\rA(f)$.  If $R>0$, there exist $m\in\{2,3\}$ and points
\(
 \xi_1,\cdots,\xi_m\in\overline{\omega_f(\D)}
\)
with $d_{\D}(c,\xi_j)=R$ such that $c$ is already the circumcenter of $\{\xi_1,\cdots,\xi_m\}$.  If $m=2$, then $c$ is the hyperbolic midpoint of $\xi_1$ and $\xi_2$.  If $m=3$, the outward unit tangent vectors at $c$ toward the $\xi_j$ have the origin in their Euclidean convex hull in $T_c\D$.
\end{theorem}

\begin{theorem}
\label{thm:intro-jung}
For every orientation-preserving harmonic mapping $f$ in $\D$ with $\norm{\omega_f}_\infty<1$,
\begin{equation}\label{eq:intro-jung}
 \frac{\dA(f)}{2}
 \le \rA(f)
 \le
 J\bigl(\dA(f)\bigr),
 \quad
 J(t)=\arcsinh\!\left(\frac{2}{\sqrt3}\sinh\frac t2\right).
\end{equation}
Both inequalities are sharp within the class of globally univalent harmonic quasiconformal mappings.
\end{theorem}

The harmonic pre-Schwarzian $P_f$ is invariant under target-affine
changes; see \cite{Graf2016,HernandezMartin2015}. For the analytic part $H$
of the normalized balanced map, the next estimate uses the smallest
possible dilatation parameter in the affine orbit.

\begin{theorem}
\label{thm:intro-ps}
Let $f$ be an orientation-preserving harmonic mapping in $\D$ with $\norm{\omega_f}_\infty<1$, and let
\(
 F=\Bal[f]=H+\overline G
\)
be the canonical balanced representative defined in Section \ref{sec:canonical}.  Put
\[
 k=\norm{G'/H'}_\infty=\kA(f)<1.
\]
Then
\begin{equation}\label{eq:intro-ps}
 \norm{P_H-P_f}
 =\norm{P_H-P_F}
 \le \mathcal C(k),
\end{equation}
where, for $0<k<1$,
\begin{equation}\label{eq:Ck-intro}
\mathcal C(k)
 =\frac{\sqrt{x_k}(k^2-x_k)}{k(1-x_k)}
 =\frac{2x_k^{3/2}}{k(1+x_k)},
 \end{equation}
where $x_k\in(0,k^2)$ is given by
\begin{equation}\label{eq:xk-intro}
 x_k=\frac{3-k^2-\sqrt{(1-k^2)(9-k^2)}}{2},
\end{equation}
and $\mathcal C(0)=0$.  The estimate is sharp for every $0<k<1$, even among globally univalent harmonic quasiconformal mappings.  Moreover,
\begin{equation}\label{eq:Ck-asymptotic-intro}
 \mathcal C(k)=\frac{2}{3\sqrt3}k^2+O(k^4)
 \quad(k\to0).
\end{equation}
For any orientation-preserving target-affine mapping $A$, the same universal estimate holds with the parameter $\norm{\omega_{A\circ f}}_\infty$ in place of $k$, and balancing minimizes this parameter over the affine orbit.
\end{theorem}

In the balanced gauge, the minimum is attained simultaneously for the
maximal eccentricity and the modulus supremum of the dilatation. The
equivalence of these two optimization problems will be used to estimate
the pre-Schwarzian of the analytic part. The quadratic order of that estimate comes from the scaled form of the Schwarz–Pick lemma. Balancing gives the smallest
parameter in the resulting family of sharp bounds.
Vanishing affine diameter characterizes target-affine images of
analytic mappings, and fixed affine diameter gives bounded distortion
in the balanced gauge.

The paper is organized as follows. Section \ref{sec:prelim} reviews harmonic mappings, target-affine actions, and hyperbolic geometry.  In Section \ref{sec:invariants}, we introduce the affine circumradius, affine diameter, and pseudohyperbolic oscillation and prove their invariance.  Section \ref{sec:optimal} gives the proof of Theorem \ref{thm:intro-optimal} and constructs the balanced gauge. Section \ref{sec:support-jung} proves Theorem \ref{thm:intro-support} and the inequality assertion of Theorem \ref{thm:intro-jung}.  Section \ref{sec:examples} supplies the extremal constructions completing the sharpness proof of Theorem \ref{thm:intro-jung}.  In Section \ref{sec:canonical}, we construct a canonical section of the affine quotient and introduce an affine-stable hierarchy of harmonic mappings with uniformly bounded differential distortion. Section \ref{sec:preschwarzian} proves Theorem \ref{thm:intro-ps}. 
Section \ref{sec:further} collects further geometric consequences that clarify the geometry of the new invariants. Finally, Section \ref{sec:open} presents open problems arising from these invariants as natural continuations of the classical ALIF program.

\section{Preliminaries}\label{sec:prelim}

\subsection{Orientation-preserving harmonic mappings}

Let $\Omega\subset\C$ be a simply connected domain.  A complex-valued harmonic function $f$ in $\Omega$ can be written as
\(f=h+\overline g\),
where $h$ and $g$ are analytic functions.  The decomposition is unique up to an additive constant if one fixes, for example, $g(z_0)=0$ at a base point.  The Jacobian is
\(
 J_f=|h'|^2-|g'|^2.
\)
By Lewy's theorem~\cite{Lewy1936}, a harmonic homeomorphism has nonvanishing Jacobian.  We call $f$ orientation-preserving if $J_f>0$.  Then $h'\ne0$ and the second complex dilatation
\begin{equation}\label{eq:omega-def}
 \omega_f=\frac{g'}{h'}
\end{equation}
is analytic with $|\omega_f|<1$.

The maximal and minimal stretches of the differential are
\[
 \Lambda_f=|h'|+|g'|,
 \quad
 \lambda_f=|h'|-|g'|,
\]
so the conformal dilatation is
\begin{equation}\label{eq:pointwise-K}
 K_f(z)=\frac{\Lambda_f(z)}{\lambda_f(z)}
 =\frac{1+|\omega_f(z)|}{1-|\omega_f(z)|}.
\end{equation}
We write
\begin{equation}\label{eq:global-K}
 K(f)=\sup_{z\in\D}K_f(z)
 =\frac{1+k(f)}{1-k(f)},
 \quad k(f)=\norm{\omega_f}_\infty.
\end{equation}
We set $K(f)=+\infty$ if $k(f)=1$. If $k(f)<1$ and $f$ is a homeomorphism, $K(f)$ is the quasiconformal constant.

We will often use the family $\SH$ of orientation-preserving univalent
harmonic mappings $f=h+\overline g$ of $\D$, normalized by
$f(0)=0$ and $h'(0)=1$,
and its standard compact subfamily
\[
 \SHo=\{f=h+\overline g\in\SH:g'(0)=0\}.
\]
Different normalization conventions occur in the literature; none of the affine invariants introduced below depend on these choices.

\subsection{Target-affine transformations}

For an orientation‑preserving real‑affine mapping $A$ given by \eqref{eq:A-def},
if $f=h+\overline{g}$, then
\(
A\circ f = H+\overline{G},
\)
where, up to additive constants,
\begin{equation}\label{eq:H-G-A}
 H=ah+bg,
 \quad
 G=\overline a\,g+\overline b\,h.
\end{equation}
Consequently, we have
\begin{equation}\label{eq:omega-transform}
 \omega_{A\circ f}
 =\frac{G'}{H'}
 =\frac{\overline a\,\omega_f+\overline b}{b\,\omega_f+a}.
\end{equation}

\begin{lemma}\label{lem:affine-aut}
For every $A$ as in \eqref{eq:A-def}, the mapping
\begin{equation}\label{eq:T-A}
 T_A(\zeta)=\frac{\overline a\,\zeta+\overline b}{b\,\zeta+a}
\end{equation}
is an automorphism of $\D$.  Conversely, every automorphism of $\D$ is $T_A$ for some orientation-preserving real-affine mapping $A$.
\end{lemma}
\begin{proof}
For $q=b/a$, the formula is \[T_A(\zeta)=\frac{\overline a}{a}\frac{\zeta+\overline q}{1+q\zeta}\] with $|q|<1$. To realize \[e^{i\theta}\frac{\zeta-c}{1-\overline c\,\zeta},\] choose $a\ne0$ with $\overline a/a=e^{i\theta}$ and set $b=-a\overline c$.
\end{proof}

All subsequent analysis rests on identity \eqref{eq:omega-transform}.  Notice also that target similarities $S(w)=\lambda w+c$ with $\lambda\ne0$, induce rotations
\[
 \omega_{S\circ f}=\frac{\overline\lambda}{\lambda}\omega_f.
\]

If $\phi\in\Aut(\D)$, then
\begin{equation}\label{eq:precomp-omega}
 \omega_{f\circ\phi}=\omega_f\circ\phi.
\end{equation}
Thus precomposition by a disk automorphism leaves the dilatation range unchanged as a set.  The additional target normalization used in a harmonic Koebe transform only rotates that range.  This will give automatic linear invariance of our global quantities.

\subsection{The hyperbolic disk}

We equip $\D$ with the Poincar\'e metric of curvature $-1$, i.e.,
\begin{equation}\label{eq:poincare-line}
 ds_{\D}=\frac{2|dz|}{1-|z|^2}.
\end{equation}
Its distance is
\begin{equation}\label{eq:poincare-distance}
 d_{\D}(z_1,z_2)
 =2\arctanh \delta(z_1,z_2),
 \quad
 \delta(z_1,z_2)=\left|\frac{z_1-z_2}{1-\overline{z_2}\, z_1}\right|,
\end{equation}
where \(z_1,z_2\in\mathbb D\), and \(\delta\) is the pseudohyperbolic distance.
 Every disk automorphism is an isometry for $d_{\D}$.

For a nonempty set $E\subset\D$, define
\begin{align}
 r_E(c)&:=\sup_{\zeta\in E}d_{\D}(c,\zeta),\label{eq:rEc}\\
 \rad_{\D}(E)&:=\inf_{c\in \D}r_E(c),\label{eq:radE}\\
 \diam_{\D}(E)&:=\sup_{\zeta,\eta\in E}d_{\D}(\zeta,\eta).
\end{align}
A point attaining the infimum in \eqref{eq:radE} is a circumcenter.  The hyperbolic disk is a complete CAT$(-1)$ space, hence a complete CAT$(0)$ space.  Thus every bounded nonempty set has a unique circumcenter; see, for example, Bridson–Haefliger~\cite[Chapter II]{BridsonHaefliger1999}.



By using the triangle inequality, isometry and uniqueness, continuity of distance, and geodesic convexity of balls, respectively, we get the following result.

\begin{lemma}\label{lem:rad-basic}
Let $E\subset\D$ be nonempty and hyperbolically bounded, with circumradius $R$ and circumcenter $c$.  Then
\begin{enumerate}[label=\textup{(\roman*)}]
\item $\diam_{\D}(E)\le2R$;
\item if $T\in\Aut(\D)$, then
$\rad_{\D}(T(E))=R$,
$\diam_{\D}(T(E))=\diam_{\D}(E)$,
and the circumcenter of $T(E)$ is $T(c)$;
\item $\rad_{\D}(E)=\rad_{\D}(\overline E)$ and
$\diam_{\D}(E)=\diam_{\D}(\overline E)$;
\item $\rad_{\D}(E)=\rad_{\D}(\operatorname{hconv}E)$, where $\operatorname{hconv}$ denotes the hyperbolic geodesic convex hull.
\end{enumerate}
\end{lemma}

For the diameter, the convex hull can also be inserted without changing the value.  This is a standard consequence of Busemann convexity in CAT$(0)$ spaces. We state the specialized version required for our purposes.
For every bounded \(E\subset D\), we have \(\operatorname{diam}_D(\operatorname{hconv} E)=\operatorname{diam}_D(E)\). If $S$ has diameter at most $D$, geodesic convexity of distance bounds the distance from any point of $S$ to every segment with endpoints in $S$ by $D$. Repeating this convexity argument as the second point ranges over another such segment controls pairwise distances within the union of such segments by $D$; iteration then gives the bound for \(\operatorname{hconv} E\).

\subsection{Pre-Schwarzian and Schwarzian derivatives}
For a locally univalent and orientation-preserving harmonic mapping $f=h+\overline{g}$, Hern\'andez--Mart\'in~\cite{HernandezMartin2015} defined the pre-Schwarzian derivative
\begin{equation}\label{eq:harmonic-P}
 P_f:=(\log J_f)_z
 =\frac{h''}{h'}-
 \frac{\overline\omega\,\omega'}{1-|\omega|^2}.
\end{equation}
The associated Schwarzian derivative is
\begin{equation}\label{eq:harmonic-S}
 S_f:=(P_f)_z-\frac12P_f^2.
\end{equation}
Both enjoy the expected chain rules under analytic precomposition, and both are invariant under orientation-preserving target-affine mappings; see~\cite{Graf2016,HernandezMartin2015}.  We use the norm
\begin{equation}\label{eq:Pnorm}
 \norm{P_f}=\sup_{z\in\D}(1-|z|^2)|P_f(z)|.
\end{equation}
The affine invariance of $P_f$ will be particularly useful because the analytic part of $f$ changes under affine balancing while $P_f$ does not.
\vskip.20in
\section{Affine radius, affine diameter, and affine oscillation}\label{sec:invariants}

We now introduce the main global invariants of the paper.

\begin{definition}\label{def:aff-invariants}
For an orientation-preserving harmonic mapping $f$ in $\D$, set $E_f=\overline{\omega_f(\D)}$, the closure in $\D$. The \emph{affine radius} and \emph{affine diameter} are
\begin{equation}\label{eq:rA}
 \rA(f)=\rad_{\D}(E_f),\quad \dA(f)=\diam_{\D}(E_f).
\end{equation}
When $E_f$ is bounded, it is compact in $\D$, and its circumcenter
\begin{equation}\label{eq:cA}
 \cA(f)=c_{\mathrm{aff}}(E_f)
\end{equation}
is the \emph{affine circumcenter}. We set
\begin{align}
 \kA(f)&=\tanh\frac{\rA(f)}{2},\label{eq:kA-KA1}\\
 \KA(f)&=\frac{1+\kA(f)}{1-\kA(f)}=e^{\rA(f)},\label{eq:kA-KA}\\
 \qA(f)&=\tanh\frac{\dA(f)}{2},\label{eq:qA}
\end{align}
where $\qA(f)$ is the \emph{pseudohyperbolic affine oscillation} and the formulas extend to infinite radius or diameter by limits.
\end{definition}
By \eqref{eq:poincare-distance}, we have
\begin{equation}\label{eq:qA-explicit}
 \qA(f)
 =\sup_{z,w\in\D}
 \left|
 \frac{\omega_f(z)-\omega_f(w)}
 {1-\overline{\omega_f(w)}\,\omega_f(z)}
 \right|.
\end{equation}
Thus $\qA$ is a M\"obius-invariant oscillation of the complex dilatation.

\begin{remark}
If $\norm{\omega_f}_\infty<1$, then $E_f$ is compact in $\D$. Its radius
and diameter agree with those of $\omega_f(\D)$ by
Lemma~\ref{lem:rad-basic}(iii), while compactness ensures the existence
of the contacts used in Section~\ref{sec:support-jung}.
\end{remark} 

\begin{proposition}
\label{prop:full-invariance}
Let $f$ be as in Definition \ref{def:aff-invariants}. 
\begin{enumerate}[label=\textup{(\roman*)}]
\item If $A$ is any orientation-preserving real-affine mapping of the target, then
\[
 \rA(A\circ f)=\rA(f),\ \
 \dA(A\circ f)=\dA(f),\ \
 \KA(A\circ f)=\KA(f),\ \
 \qA(A\circ f)=\qA(f).
\]
If $\rA(f)<\infty$, then
\[
 \cA(A\circ f)=T_A(\cA(f)).
\]
\item If $\phi\in\Aut(\D)$, then
\[
 \rA(f\circ\phi)=\rA(f),\quad
 \dA(f\circ\phi)=\dA(f).
\]
The same holds for normalized harmonic Koebe transforms.
\end{enumerate}
\end{proposition}
\begin{proof}
By \eqref{eq:omega-transform}, we have $E_{A\circ f}=T_A(E_f)$, and
Lemma~\ref{lem:affine-aut} gives a hyperbolic isometry. This preserves
the defining suprema and infima, including infinite values, and
transports the center when the radius is finite. Precomposition
preserves $E_f$ by \eqref{eq:precomp-omega}, and a Koebe normalization
only rotates it.
\end{proof}

The next observation characterizes the zero level of this hierarchy.
For an orientation-preserving harmonic mapping \(f = h + \overline{g}\) on \(\mathbb{D}\), the conditions \(d_{\text{aff}}(f)=0\), \(r_{\text{aff}}(f)=0\), constancy of \(\omega_f\), and the existence of an orientation-preserving real-affine map \(A\) for which \(A\circ f\) is analytic are equivalent. The first three conditions mean that \(E_f\) is a singleton. Given constant \(\omega_f\equiv c\), Lemma  \ref{lem:affine-aut} yields \(A\) with \(T_A(c)=0\); in the reverse direction, \(A\circ f\) being analytic implies \(T_A\circ\omega_f=0\), and injectivity of \(T_A\) returns that \(\omega_f\) is constant.

We observe that the usual number $k(f)=\norm{\omega_f}_\infty$ measures the distance of the dilatation range from the distinguished origin of the disk model.  By contrast, $\rA(f)$ measures the radius of the smallest hyperbolic disk containing the range, with the center allowed to move.  Thus $\rA$ removes precisely the non-intrinsic choice introduced by the target affine gauge.
For an orientation-preserving harmonic mapping, the uniform boundedness of the differential distortion is exactly the condition that the dilatation range be hyperbolically bounded.  If the mapping is also globally univalent, this is equivalent to quasiconformality onto its image.
Moreover, the conditions \(\|\omega_f\|_\infty<1\), \(r_{\text{aff}}(f)<\infty\), \(d_{\text{aff}}(f)<\infty\), and \(q_{\text{aff}}(f)<1\) are equivalent. The first condition is hyperbolic boundedness of \(E_f\), as \(d_\mathbb{D}(0,\zeta)=2\operatorname{arctanh}|\zeta|\). By the triangle inequality, hyperbolic boundedness is equivalent to finiteness of the radius or diameter, while \eqref{eq:qA} yields the last equivalence; for univalent $f$, we use \eqref{eq:global-K}.

\vskip.20in
\section{Optimal affine balancing}\label{sec:optimal}

We now prove the central minimization theorem.
To formulate this result, we first introduce the notion of affine-balanced harmonic mappings.

\begin{definition}
An orientation-preserving harmonic mapping $f$ with uniformly bounded differential distortion is called \emph{affine-balanced}, or simply \emph{balanced}, if
\(
 \cA(f)=0.
\)
Equivalently, the smallest hyperbolic disk containing $E_f$ is centered at $0$.
\end{definition}

\begin{figure}[htbp]
\centering
\begin{tikzpicture}[x=1.8cm,y=1.8cm,font=\small]
  \begin{scope}[shift={(-1.62,0)}]
    \draw[gray] (0,0) circle (1);
    \fill[blue!10] (.343,0) circle (.448);
    \draw[blue!65!black,thick] (.343,0) circle (.448);
    \fill (0,0) circle (1pt) node[below] {$0$};
    \fill[blue!65!black] (.45,0) circle (1.3pt) node[above] {$c$};
    \node[below] at (0,-1.06) {$E_f$};
  \end{scope}
  \draw[->,thick] (-.47,0) -- (.47,0) node[midway,above] {$T_{A_c}$};
  \begin{scope}[shift={(1.62,0)}]
    \draw[gray] (0,0) circle (1);
    \fill[blue!10] (0,0) circle (.53);
    \draw[blue!65!black,thick] (0,0) circle (.53);
    \fill[blue!65!black] (0,0) circle (1.3pt) node[above] {$0$};
    \node[below] at (0,-1.06) {$E_{A_c\circ f}$};
  \end{scope}
\end{tikzpicture}
\caption{Affine balancing of an off-center hyperbolic disk. The induced disk
automorphism sends its hyperbolic center $c$ to $0$ and preserves its radius.}
\label{fig:fig1_minimization}
\end{figure}

The shift of the hyperbolic disk’s center under this automorphism is illustrated in Figure~\ref{fig:fig1_minimization}.
Equipped with this geometric intuition, we now turn to the proof of Theorem~\ref{thm:intro-optimal}.

\begin{proof}[Proof of Theorem~\ref{thm:intro-optimal}]
Write $E=E_f$, which is compact because $k(f)<1$. For each affine map $A$, \eqref{eq:global-K} and \eqref{eq:omega-transform} yield
\begin{equation}\label{eq:logK-hyp}
 \log K(A\circ f)
 =\max_{\zeta\in E}d_{\D}(0,T_A(\zeta))
 =r_E(T_A^{-1}(0)).
\end{equation}
Here \[d_{\D}(0,u)=\log\frac{1+|u|} {1-|u|},\] and the maximum over $E$ equals the supremum over $\omega_f(\D)$ by continuity. The second equality uses the isometry $T_A^{-1}$. Conversely, every candidate center $p\in\D$ corresponds to the orientation-preserving affine map
\begin{equation}\label{eq:Ac}
 A_p(w)=w-\overline p\,\overline w,
\end{equation}
whose determinant is $1-|p|^2$ and whose action is
\begin{equation}\label{eq:phi-c}
 T_{A_p}(\zeta)=\frac{\zeta-p}{1-\overline p\,\zeta}.
\end{equation}
Thus minimizing $\log K(A\circ f)$ over all affine maps is exactly minimizing $r_E(p)$ over $\D$.

For clarity, the minimization of $r_E$ can be seen directly. It is $1$-Lipschitz by the triangle inequality. Fix $\zeta_0\in E$. Every sublevel \[\{p:r_E(p)\le M\}\] lies in the compact closed hyperbolic ball of radius $M$ about $\zeta_0$, so a minimizing sequence has a convergent subsequence and the limit is a minimizer. For two candidates $p_1,p_2$ with midpoint $m$, for each $\zeta\in E$, the CAT$(0)$ midpoint inequality gives
\[
 d^2_{\D}(m,\zeta)
 \le \frac{d^2_{\D}(p_1,\zeta)+d^2_{\D}(p_2,\zeta)}{2}
      -\frac{d^2_{\D}(p_1,p_2)}{4}.
\]
Taking the supremum proves that two distinct points with the same minimal value would have a midpoint of smaller value. Hence the minimizer is $c=\cA(f)$, and \eqref{eq:logK-hyp} gives \eqref{eq:optimal-main}. It also shows that $A$ minimizes precisely when $T_A(c)=0$, so $A_c\circ f$ is balanced.

In coefficient form, the centering condition is
\[
 T_A(c)=0
 \Longleftrightarrow
 \overline a\,c+\overline b=0
\Longleftrightarrow b=-a\overline c.
\]
The denominator of $T_A$ cannot vanish at $c$ since $|b|<|a|$ and
$|c|<1$. In particular, every minimizer is of the form
\[
 A(w)=a\bigl(w-\overline c\,\overline w\bigr)+\gamma,
 \quad a\ne0,\ \ \gamma\in\C.
\]
Let $k_c=\tanh(\rA(f)/2)$. The closure of the dilatation range after
balancing lies in the Euclidean disk of radius $k_c$. Some point of
$E_f$ is at maximum distance from $c$, and its image has modulus
$k_c$. Thus the norm of the balanced dilatation is exactly $k_c$,
even when only points of the closure attain it.

The transformation law \eqref{eq:omega-transform} applied to
successive postcompositions gives $T_{B\circ A}=T_B\circ T_A$.
Indeed, the linear harmonic map $z+\overline{\zeta z}$ has dilatation
$\zeta$, so this yields the identity at every $\zeta\in\D$.
If \[A_j(w)=a_jA_c(w)+\gamma_j\] for $j=1,2$, the composition
$A_2\circ A_1^{-1}$ is the similarity
\[
 w\longmapsto \frac{a_2} {a_1}(w-\gamma_1)+\gamma_2.
\]
Conversely, every orientation-preserving similarity acts on the
dilatation disk by a rotation about $0$ and therefore preserves
minimality. This also covers constant dilatation, when the minimum is $K=1$. If $f$ is globally univalent, all these affine representatives are homeomorphisms onto their images, so their maximal differential distortion is the quasiconformal constant.
\end{proof}

\begin{example}\label{42}
For
\begin{equation}\label{63}
 f_k(z)=z+\frac{k}{2}\overline z^{\,2},\quad 0<k<1,
\end{equation}
the range $\omega_{f_k}(\D)=k\D$ is centered at zero, so $f_k$ is balanced and \[K(f_k)=\KA(f_k)=\frac{1+k}{1-k}.\] Univalence follows from
\[|f_k(z)-f_k(w)|\ge(1-k)|z-w|.\]
\end{example}

\begin{example}
For $0<k<1$ and $a\in\D\setminus\{0\}$, take
\[\omega(z)=\frac{kz+a}{1+\overline a kz}\] and
\[f(z)=z+\overline{\int_0^z\omega(\zeta)\,d\zeta}.\]
Then $\omega(\D)$ is the hyperbolic disk of radius $2\arctanh k$ centered at $a$. Its closure lies in $\D$, so $\norm{\omega}_\infty<1$ and the segment estimate
\[|f(z)-f(w)|\ge (1-\norm{\omega}_\infty)|z-w|\]
gives univalence. Theorem~\ref{thm:intro-optimal} gives
\[K(A_a\circ f)=\KA(f)<K(f).\]
\end{example}

By Theorem~\ref{thm:intro-optimal}, if $f$ is balanced, then
\[
\|\omega_f\|_\infty = k_{\text{aff}}(f) = \tanh\frac{r_{\text{aff}}(f)}{2}\tag{4.5}
\] and
\[
K(f) = K_{\text{aff}}(f) = e^{r_{\text{aff}}(f)}. \tag{4.6}
\]
Moreover, every affine representative $F = A\circ f$ satisfies $K(F)\ge K_{\text{aff}}(f)$.

The minimization admits a useful interpretation.  The disk coordinate $0$ corresponds to the Euclidean conformal structure in the target.  An affine mapping changes that reference conformal structure.  The balanced gauge chooses the reference structure minimizing the worst hyperbolic distance to all conformal structures encoded by $\omega_f(z)$.  In this sense, $\cA(f)$ is a Chebyshev center in the Teichm\"uller disk of linear conformal structures.
The affine circumcenter obeys \(c_{\text{aff}}(A\circ f)=T_A(c_{\text{aff}}(f))\), which is the center claim of Proposition~\ref{prop:full-invariance}.

\vskip.20in
\section{Support points and the hyperbolic Jung theorem}\label{sec:support-jung}

We use the first variation to prove the finite support property.
Figure~\ref{fig:three_point_support} illustrates the two minimal cases.

\begin{figure}[htbp]
\centering
\begin{tikzpicture}[x=1.75cm,y=1.75cm,font=\small]
  \begin{scope}[shift={(-1.55,0)}]
    \draw[blue!65!black,dashed] (0,0) circle (.85);
    \draw[thick] (-.85,0) -- (.85,0);
    \fill (0,0) circle (1.2pt) node[above] {$c$};
    \fill (-.85,0) circle (1.5pt) node[left] {$\xi_1$};
    \fill (.85,0) circle (1.5pt) node[right] {$\xi_2$};
    \node[below] at (0,-.95) {Diametral pair};
  \end{scope}
  \begin{scope}[shift={(1.55,0)}]
    \draw[blue!65!black,dashed] (0,0) circle (.85);
    \foreach \a/\lab in {24/1,138/2,249/3} {
      \draw[thick] (0,0) -- (\a:.85);
      \fill (\a:.85) circle (1.5pt) node[shift={(\a:9pt)}] {$\xi_{\lab}$};
    }
    \fill (0,0) circle (1.2pt) node[above] {$c$};
    \node[below] at (0,-.95) {Three contact points};
  \end{scope}
\end{tikzpicture}
\caption{The two minimal types of support in Theorem~\ref{thm:intro-support},
drawn after sending $c$ to the origin. In the right-hand example all
successive angles are less than $\pi$.}
\label{fig:three_point_support}
\end{figure}

\subsection{A support lemma in the hyperbolic plane}

Let $E\subset\D$ be nonempty and compact, let $c$ be its circumcenter, and let $R=\rad_{\D}(E)>0$.  Define the contact set
\begin{equation}\label{eq:contact-set}
 \Gamma(E):=\{\zeta\in E:d_{\D}(c,\zeta)=R\}.
\end{equation}
For $\zeta\in\Gamma(E)$, let $u_\zeta\in T_c\D$ be the unit tangent vector at $c$ pointing along the geodesic from $c$ to $\zeta$.

\begin{lemma}
\label{lem:contact-balance}
With the notation above, we have
\begin{equation}\label{eq:contact-hull}
 0\in\conv\{u_\zeta:\zeta\in\Gamma(E)\}
 \subset T_c\D\cong\R^2.
\end{equation}
\end{lemma}
\begin{proof}
 Suppose not. Then strict separation of the compact convex hull yields a unit vector \(\boldsymbol v\) and a constant \(\eta>0\) such that
\begin{equation}\label{eq:strict-sep}
 \ip{u_\zeta}{v}\ge\eta\quad(\zeta\in\Gamma(E)).
\end{equation}
Put \[F(t,\zeta)=d_{\D}(\exp_c(tv),\zeta).\] Since $R>0$, distance is smooth in the first variable near the contact set and the first variation formula gives
\begin{equation}\label{eq:first-var}
 \partial_tF(0,\zeta)=-\ip{u_\zeta}{v}\le-\eta
 \quad(\zeta\in\Gamma(E)).
\end{equation}
Compactness and continuity give a relative neighborhood $U$ of $\Gamma(E)$ in $E$ and $\tau>0$ on which $\partial_t F\le-\eta/2$ for $0\le t\le\tau$. Thus \[F(t,\cdot)\le R-\frac{t\eta}{2}\] on $U$. If $E\setminus U$ is nonempty, its compactness gives \[F(0,\cdot)\le R-\varepsilon\] for some $\varepsilon>0$, and the triangle inequality gives \[F(t,\cdot)\le R-\varepsilon+t.\] Taking $t>0$ small enough contradicts the minimality of $R$.
\end{proof}

\begin{theorem}
\label{thm:support}
Let $E\subset\D$ be nonempty and compact, with circumcenter $c$ and radius $R>0$.  Then there exist $m\in\{2,3\}$ and contact points $\xi_1,\cdots,\xi_m\in\Gamma(E)$ such that $c$ is the circumcenter of $\{\xi_1,\cdots,\xi_m\}$.
More precisely, we have
\begin{enumerate}[label=\textup{(\roman*)}]
\item if $m=2$, then the tangent directions $u_{\xi_1}$ and $u_{\xi_2}$ are opposite, so $c$ is the midpoint of the geodesic segment $[\xi_1,\xi_2]$ and
\(
 d_{\D}(\xi_1,\xi_2)=2R;
\)
\item if $m=3$, then there are positive numbers $\lambda_1,\lambda_2,\lambda_3$ with $\lambda_1+\lambda_2+\lambda_3=1$ such that
\begin{equation}\label{eq:three-balance}
 \lambda_1u_{\xi_1}+\lambda_2u_{\xi_2}+\lambda_3u_{\xi_3}=0.
\end{equation}
\end{enumerate}
\end{theorem}
\begin{proof}
Apply the Carath\'eodory's theorem~\cite[Theorem 17.1]{Rockafellar1970} in $T_c\D\cong\R^2$ to \eqref{eq:contact-hull}, and choose a subcollection of minimal cardinality whose convex hull contains $0$. One unit direction cannot suffice. Two directions suffice precisely when they are opposite; their endpoints are at distance $R$ on opposite sides of $c$ on one geodesic. If three are needed, the triangle of their directions contains the origin in its interior, giving the positive weights in \eqref{eq:three-balance}.

In either case, the weights can be taken positive with sum $1$ and weighted direction sum $0$; for two directions take weights $1/2$. To verify that the contact points determine the original ball, fix $p\ne c$ and let $\gamma$ be the unit-speed geodesic from $c$ to $p$ of length $L$. For each $j$, distance to $\xi_j$ along $\gamma$ is convex. The slope of its secant from $0$ to $L$ is at least its right derivative at $0$, given by first variation. Thus
\[
 d_{\D}(p,\xi_j)\ge R-L\ip{u_{\xi_j}}{\gamma'(0)}.
\]
Upon multiplying by the weights and summing, the linear terms cancel, so that
\[
 \max_j d_{\D}(p,\xi_j) \ge
 \sum_j \lambda_j d_{\D}(p,\xi_j) \ge R.
\]
Thus every candidate center has covering radius at least $R$. At $c$,
each selected point is at distance exactly $R$, so this lower bound
is attained. The CAT$(0)$ uniqueness argument in Section~\ref{sec:optimal}
applies to the finite set and identifies the minimizing center with $c$.
\end{proof}

In what follows, we provide the proof of Theorem~\ref{thm:intro-support}.

\begin{proof}[Proof of Theorem~\ref{thm:intro-support}]
Since $k(f)<1$, the set $E_f$ is compact in $\D$. The hypothesis $R>0$ makes $\omega_f$ nonconstant, so $\omega_f(\D)$ is open by the open mapping theorem. No point of this open image can be a contact point: its neighborhood would contain points farther from $c$ on the geodesic ray from $c$. Thus $\Gamma(E_f)$ consists of boundary points of the dilatation range.
Each is the limit of dilatation values $\omega_f(z_n)$. For any such sequence \(z_n\), compactness of closed subdisks and continuity of \(\omega_f\) imply \(|z_n|\to 1\), since an interior cluster point would place the contact value in the open image. Apply Theorem~\ref{thm:support} to
$E_f$. Its selected contact points have the same minimum covering
radius as the full range, and its two alternatives give the midpoint
and convex hull assertions.
\end{proof}

\subsection{Jung's inequality and sharp affine distortion bounds}

We use the sharp hyperbolic Jung theorem in dimension two.  In our curvature $-1$ convention it states that if a bounded subset $E$ of the hyperbolic plane has diameter at most $D$, then
\begin{equation}\label{eq:Dekster}
 \rad_{\D}(E)
 \le J(D),
 \quad
 J(D)=\arcsinh\!\left(\frac{2}{\sqrt3}\sinh\frac D2\right).
\end{equation}
This is the case $n=2$ of Dekster~\cite{Dekster1995}; see also the formulation in 
~\cite[Theorem 3.3]{NasserRainioVuorinen2022}.  The constant is sharp and corresponds to a regular equilateral hyperbolic triangle. Figure~\ref{fig:thm1_3_diam_radius} shows the two bounds. The finite support result also gives the following direct proof in dimension two.

The following theorem contains the inequality part of Theorem~\ref{thm:intro-jung} and records the equivalent estimate for $\KA$.  The sharpness assertion in Theorem~\ref{thm:intro-jung} will be completed in Section \ref{sec:examples}.

\begin{theorem}
\label{thm:affine-jung}
For every orientation-preserving harmonic mapping $f$ with $\norm{\omega_f}_\infty<1$,
\begin{equation}\label{eq:affine-jung}
 \frac{\dA(f)}2
 \le \rA(f)
 \le J(\dA(f)).
\end{equation}
Equivalently,
\begin{equation}\label{eq:KA-jung}
 \exp\left(\frac{\dA(f)}{2}\right)
 \le \KA(f)
 \le \exp(J(\dA(f))).
\end{equation}
Both bounds are sharp among globally univalent harmonic quasiconformal mappings.
\end{theorem}
\begin{proof}[Proof of the inequalities in Theorem~\ref{thm:intro-jung}]
Put $D=\dA(f)$ and $R=\rA(f)$. The triangle inequality gives $D\le2R$, and if $R=0$ both bounds are equalities. For $R>0$, take the support from Theorem~\ref{thm:support}. For two support points, $D=2R$ and $R\le J(D)$ holds because $2/\sqrt3>1$.

For three support points from a minimal subcollection, order their directions cyclically. Each gap between successive directions is less than $\pi$: a gap of at least $\pi$ would put all three directions in a closed half-plane through $0$, contradicting that $0$ lies in the interior of their triangle. The gaps sum to $2\pi$, so one gap $\theta$ lies in $[2\pi/3,\pi)$. Let $d$ be the distance between the corresponding support points. The hyperbolic law of cosines in their triangle with $c$ yields
\[
 \cosh d=\cosh^2R-\sinh^2R\cos\theta
        =1+2\sinh^2R\sin^2(\theta/2).
\]
Taking square roots in \[\cosh d-1=2\sinh^2(d/2)\] gives
\[
 \sinh(D/2)\ge\sinh(d/2)=\sinh R\sin(\theta/2)
 \ge\frac{\sqrt3}{2}\sinh R .
\]
Since $\sinh$ is increasing, this is $R\le J(D)$. If $R=J(D)>0$,
equality in the first inequality forces $d=D$, and equality in the
last forces the largest gap to be $2\pi/3$. All three gaps are then
$2\pi/3$ because they sum to $2\pi$. The cosine rule gives side length
$D$ for every side of the resulting equilateral support triangle. The inequalities for $\KA$ follow by exponentiation. The sharpness constructions, including realization of the upper extremal by a univalent harmonic map, are given in Section~\ref{sec:examples}.
\end{proof}

\begin{figure}[htbp]
\centering
\begin{tikzpicture}[x=2.25cm,y=2.22cm,font=\small]
  \def\JungSamples{
    (0.000000,0.0000000) (0.050000,0.0288665) (0.100000,0.0577270) (0.150000,0.0865756) (0.200000,0.1154063)
    (0.250000,0.1442133) (0.300000,0.1729912) (0.350000,0.2017345) (0.400000,0.2304379) (0.450000,0.2590967)
    (0.500000,0.2877060) (0.550000,0.3162616) (0.600000,0.3447594) (0.650000,0.3731958) (0.700000,0.4015673)
    (0.750000,0.4298709) (0.800000,0.4581041) (0.850000,0.4862643) (0.900000,0.5143498) (0.950000,0.5423587)
    (1.000000,0.5702898) (1.050000,0.5981421) (1.100000,0.6259149) (1.150000,0.6536076) (1.200000,0.6812202)
    (1.250000,0.7087527) (1.300000,0.7362054) (1.350000,0.7635788) (1.400000,0.7908737) (1.450000,0.8180910)
    (1.500000,0.8452316) (1.550000,0.8722968) (1.600000,0.8992879) (1.650000,0.9262063) (1.700000,0.9530535)
    (1.750000,0.9798311) (1.800000,1.0065408) (1.850000,1.0331843) (1.900000,1.0597634) (1.950000,1.0862798)
    (2.000000,1.1127353) (2.050000,1.1391318) (2.100000,1.1654712) (2.150000,1.1917551) (2.200000,1.2179855)
    (2.250000,1.2441642) (2.300000,1.2702929) (2.350000,1.2963735) (2.400000,1.3224075) (2.450000,1.3483968)
    (2.500000,1.3743431) (2.550000,1.4002479) (2.600000,1.4261128) (2.650000,1.4519394) (2.700000,1.4777292)
    (2.750000,1.5034838) (2.800000,1.5292044) (2.850000,1.5548925) (2.900000,1.5805495) (2.950000,1.6061766)
    (3.000000,1.6317750)
  }
  \foreach \x in {1,2,3} \draw[black!11] (\x,0) -- (\x,1.7);
  \foreach \y in {0.5,1.0,1.5} \draw[black!11] (0,\y) -- (3,\y);
  \path[fill=blue!12] plot coordinates {\JungSamples} -- (3,1.5) -- cycle;
  \draw[very thick,blue!55!black] plot coordinates {\JungSamples};
  \draw[thick,densely dashed] (0,0) -- (3,1.5);
  \draw[->] (0,0) -- (3.16,0) node[right] {$\dA(f)$};
  \draw[->] (0,0) -- (0,1.77) node[above] {$\rA(f)$};
  \node[below left] at (0,0) {$0$};
  \foreach \x in {1,2,3} \draw (\x,0.022) -- (\x,-0.022) node[below] {$\x$};
  \foreach \y in {0.5,1.0,1.5} \draw (0.022,\y) -- (-0.022,\y) node[left] {$\y$};
  \draw[very thick,blue!55!black] (0.16,1.54) -- (0.43,1.54)
    node[right,black] {$J(D)$};
  \draw[thick,densely dashed] (0.16,1.32) -- (0.43,1.32)
    node[right] {$D/2$};
\end{tikzpicture}
\caption{The bounds in Theorem~\ref{thm:intro-jung} for $D=\dA(f)$.
  A pair at distance $2\rA(f)$ gives lower equality. An equilateral contact triple of side length $D$ gives upper equality.}
\label{fig:thm1_3_diam_radius}
\end{figure}

It is sometimes preferable to express the result using the pseudohyperbolic oscillation $q=\qA(f)$.  Since
\[
 q=\tanh\frac{\dA(f)}2,
\]
we obtain a closed formula for the Jung upper bound on $\kA$.

\begin{corollary}
\label{cor:pseudohyp-jung}
Let $q=\qA(f)\in[0,1)$.  Then
\begin{equation}\label{eq:kA-lower-q}
 \kA(f)\ge \frac{q}{1+\sqrt{1-q^2}}
\end{equation}
and
\begin{equation}\label{eq:kA-upper-q}
 \kA(f)
 \le
 j(q):=
 \frac{2q}
 {\sqrt{3+q^2}+\sqrt{3(1-q^2)}}.
\end{equation}
Consequently,
\begin{equation}\label{eq:KA-q}
 \sqrt{\frac{1+q}{1-q}}
 \le \KA(f)
 \le \frac{1+j(q)}{1-j(q)}.
\end{equation}
\end{corollary}
\begin{proof}
Substitute \[\dA=2\arctanh q\] in \eqref{eq:affine-jung} and use \[\tanh\frac t2=\frac{\sinh t}{1+\cosh t}\] for both endpoints. In the upper bound, we have \[\sinh J(\dA)=\frac{2q}{\sqrt{3(1-q^2)}}.\] The bounds for $\KA$ follow from \eqref{eq:kA-KA}.
\end{proof}

For small oscillations, the Jung bounds admit a transparent Euclidean limit. 
Expand the two functions in \eqref{eq:Dekster} at zero, we get the following result.

\begin{corollary}
\label{cor:small-jung}
As $D\to0$, we have
\begin{equation}\label{eq:J-small}
 J(D)=\frac{D}{\sqrt3}+O(D^3).
\end{equation}
Then
\[
 \frac D2\le\rA(f)\le\frac D{\sqrt3}+O(D^3)
 \ \text{when }\dA(f)=D\to0.
\]
\end{corollary}

The constants \(1/2\) and \(1/\sqrt{3}\) are exactly the two-dimensional Euclidean circumradius constants for a segment-type configuration and a disk configuration, as expected from the infinitesimal flatness of the hyperbolic metric.
\vskip.20in

\section{Extremal constructions and classical normalization}\label{sec:examples}
We next show that the bounds above are genuinely sharp within harmonic geometric function theory, not merely in abstract hyperbolic geometry.
\begin{lemma}\label{lem:realization}
Let $\omega:\D\to\D$ be analytic with $\norm{\omega}_\infty\le k<1$. Then
\begin{equation}\label{eq:realization-f}
 g(z)=\int_0^z\omega(\zeta)\,d\zeta,\quad
 f(z)=z+\overline{g(z)}
\end{equation}
is orientation-preserving, globally univalent, and quasiconformal with
$K(f)\le(1+k)/(1-k)$. Its second dilatation is $\omega$.
\end{lemma}

\begin{proof}
Noting that $h'=1$ and $g'=\omega$, so $J_f=1-|\omega|^2>0$.
Integration over the segment $[z,w]\subset\D$ gives
\[|g(z)-g(w)|\le k|z-w|\] and hence
\[|f(z)-f(w)|\ge(1-k)|z-w|.\]
The distortion bound follows from \eqref{eq:pointwise-K}.
\end{proof}


We now prove the sharpness of Theorem~\ref{thm:intro-jung}.

\begin{proof}[Sharpness in Theorem~\ref{thm:intro-jung}]
For the lower bound at a prescribed diameter $D>0$, take $k=\tanh(D/4)$ and use the dilatation
\begin{equation}\label{eq:omega-kz}
 \omega(z)=kz.
\end{equation}
The realization in Lemma~\ref{lem:realization} is the polynomial $f_k$
in \eqref{63}. Its closed dilatation range is the hyperbolic disk
$\overline B_{\D}(0,R)$ with $R=2\arctanh k$. The pair $-k,k$ is
at distance $2R$, and the triangle inequality bounds all pairwise
distances by $2R$. Any enclosing disk has radius at least $R$ because
it contains this pair. Thus \[\dA(f_k)=2\rA(f_k)=2R.\] The same range,
and hence the same equality, is obtained from $\omega(z)=kz^m$ for every
integer $m\ge1$.

For the upper bound, let $D>0$ and take three equally spaced points
$\xi_1,\xi_2,\xi_3$ on a hyperbolic circle of radius $R$ centered at
$c$. The angle at $c$ between adjacent vertices is $2\pi/3$, and the
geodesic from $c$ to the midpoint of a side bisects that angle
orthogonally to the side. The hyperbolic sine rule in the resulting
right triangle gives \[\sinh\frac {D}{2}=\frac{\sqrt3}{2}\sinh R\] when the side length
is $D$. We can therefore prescribe $D$ by taking
\begin{equation}\label{eq:equilateral-R}
 R=J(D)=\arcsinh\!\left(\frac{2}{\sqrt3}\sinh\frac D2\right).
\end{equation}
The closed hyperbolic triangle \(\Delta_\mathbb{D}\), the geodesic convex hull of its three vertices, has diameter $D$, as the $h$-convex hull preserves hyperbolic diameter. Rotation through $2\pi/3$ about $c$
preserves the triangle. It must fix the unique circumcenter, and its
only fixed point is $c$. The triangle lies in $\overline B_{\D}(c,R)$
by geodesic convexity and contains all three vertices at distance $R$
from $c$. Its circumradius is therefore exactly $R$.

Place $c=0$ by a disk isometry. Then all three vertices have modulus
$k=\tanh(R/2)$, and the enclosing hyperbolic ball is the Euclidean
closed disk of radius $k<1$. The interior $U_D$ of $\Delta_D$ is a
simply connected Jordan domain with $\overline{U_D}=\Delta_D\Subset\D$.
By the Riemann mapping theorem, choose \[\phi_D:\D\to U_D\] to be conformal.
Because the interior approaches all three vertices, $\norm{\phi_D}_\infty=k$.

Define
\[
 f_D(z):=z+\overline{\int_0^z\phi_D(\zeta)\,d\zeta}.
\]
Lemma~\ref{lem:realization} shows that this map is globally univalent
and quasiconformal. Its analytic derivative is $1$ and its dilatation
is exactly $\phi_D$, so $E_{f_D}=\Delta_D$. In particular, $f_D$ is
balanced and satisfies
\[
 \dA(f_D)=D,\quad \rA(f_D)=J(D),\quad K(f_D)=e^{J(D)}.
\]
Thus the radius and distortion bounds are attained for every
prescribed positive diameter. For $D=0$, an analytic univalent map
attains both endpoints.
\end{proof}


Fix $c\in\D$ and $R>0$. Choose a target-affine representative with
dilatation range $B_{\D}(c,R)$ by translating the centered
hyperbolic disk for $f_k$, with $k=\tanh(R/2)$, via a disk
automorphism and Lemma~\ref{lem:affine-aut}. Its affine radius is $R$.
The triangle inequality, with equality along the ray from $0$ through
$c$ continued a distance $R$ beyond $c$, gives
\begin{equation}\label{eq:shifted-K}
 \log K(f)=R+d_{\D}(0,c).
\end{equation}
For $c=0$, any radial ray gives equality. Consequently, we have
\begin{equation}\label{eq:shifted-ratio}
 \frac{K(f)}{\KA(f)}=e^{d_{\D}(0,c)},
\end{equation}
which is unbounded along this single target-affine orbit as $|c|\to1$.

\begin{theorem}\label{thm:omega0-comparison}
If $f$ is an orientation-preserving harmonic mapping with $\norm{\omega_f}_\infty<1$
and $\omega_f(0)=0$, then
\begin{equation}\label{eq:sqrtK}
 \sqrt{K(f)}\le\KA(f)\le K(f).
\end{equation}
Equivalently,
\begin{equation}\label{eq:r0-R}
 \frac12\log K(f)\le\rA(f)\le\log K(f).
\end{equation}
Both coefficients in \eqref{eq:r0-R} are optimal. The lower one is
sharp in the limit, and the upper one is attained for $\omega(z)=kz$.
\end{theorem}

\begin{proof}
Since $0\in E_f$, the center $c=\cA(f)$ satisfies
\begin{equation}\label{eq:center-to-zero}
 d_{\D}(0,c)\le R:=\rA(f).
\end{equation}
The covering radius \[r_0=\sup_{\zeta\in E_f}d_{\D}(0,\zeta)=\log K(f)\]
therefore satisfies \[R\le r_0\le d_{\D}(0,c)+R\le2R.\]
The centered disk example gives $r_0=R$. For the other endpoint, fix
$R>0$ and take $c_j$ with $d_{\D}(0,c_j)<R$ tending to $R$. Realize
conformal maps onto $B_{\D}(c_j,R)$ taking $0$ to $0$. Their ranges
have circumradius $R$, while \eqref{eq:shifted-K} gives
\(\log K(f_j)\to2R\).
\end{proof}

\begin{remark}
The inequality \eqref{eq:sqrtK} explains the precise limitation of the usual point normalization.  Eliminating $\omega(0)$ can lose at most a square in the maximal differential distortion relative to the globally optimal affine gauge; in the univalent category, this is a statement about quasiconformal constants, and the loss cannot be improved uniformly.
\end{remark}

\vskip.10in
\section{A canonical section of the affine quotient}\label{sec:canonical}

The minimizer in Theorem \ref{thm:intro-optimal} is unique modulo similarities.  This makes it possible to remove the residual ambiguity by one standard normalization and obtain a canonical representative of each target-affine orbit.

\subsection{The balancing operator}

Let $f$ be an orientation-preserving harmonic mapping with $\norm{\omega_f}_\infty<1$ (so that $f_z(0)\ne0$).  Choose any minimizing affine mapping $A$ from Theorem \ref{thm:intro-optimal}, and define
\begin{equation}\label{eq:Bal-def}
 \Bal[f](z)
 =\frac{A(f(z))-A(f(0))}{(A\circ f)_z(0)}.
\end{equation}
Then
\begin{equation}\label{eq:Bal-normalization}
 \Bal[f](0)=0,
 \quad
 \Bal[f]_z(0)=1,
\end{equation}
and $\Bal[f]$ is balanced.

\begin{theorem}
\label{thm:canonical-section}
The definition \eqref{eq:Bal-def} is independent of the chosen minimizing affine mapping $A$.  Moreover, for every orientation-preserving real-affine target mapping $B$,
\begin{equation}\label{eq:Bal-affine-invariant}
 \Bal[B\circ f]=\Bal[f].
\end{equation}
Thus $\Bal$ is a canonical section of the target-affine quotient, normalized by \eqref{eq:Bal-normalization}.
\end{theorem}

\begin{proof}
By Theorem~\ref{thm:intro-optimal}, another minimizer differs from $A$
by a target similarity. Its translation cancels in the numerator of
\eqref{eq:Bal-def}, and its multiplier cancels in the quotient.
If $A$ minimizes for $f$, then $A\circ B^{-1}$ minimizes for $B\circ f$
and produces the same representative.
\end{proof}

The operator also interacts naturally with disk automorphisms.  For $\phi\in\Aut(\D)$, write
\begin{equation}\label{eq:harmonic-Koebe-general}
 \mathcal K_\phi F(z)
 =\frac{F(\phi(z))-F(\phi(0))}
 {F_z(\phi(0))\phi'(0)}.
\end{equation}
This is normalized by value and analytic derivative at the origin.

\begin{proposition}
\label{prop:Bal-Koebe}
Let $f$ be an orientation-preserving harmonic mapping with $\norm{\omega_f}_\infty<1$. For every $\phi\in\Aut(\D)$,
we have
\begin{equation}\label{eq:Bal-intertwine}
 \Bal[f\circ\phi]
 =\mathcal K_\phi(\Bal[f]).
\end{equation}
In particular, applying a Koebe transform to a balanced mapping and rebalancing produces exactly the normalized Koebe transform of its canonical affine class.
\end{proposition}

\begin{proof}
The dilatation range is unchanged by precomposition with $\phi$,
so the same $A$ minimizes for both $f$ and $f\circ\phi$. Formula
\eqref{eq:Bal-def} for $f\circ\phi$ gives \eqref{eq:harmonic-Koebe-general}
for $A\circ f$, which is also the Koebe transform of $\Bal[f]$.
\end{proof}

\subsection{An affine-stable hierarchy}

For $D\ge0$, define
\begin{equation}\label{eq:HD-class}
 \mathscr H(D)
 =\left\{
 f:
 f\ \text{is orientation-preserving and harmonic},\\
 \norm{\omega_f}_\infty<1, \dA(f)\le D
 \right\}.
\end{equation}
Let $\mathscr S(D)$ denote the univalent subfamily.  These are not normalized families; rather, they are naturally stable under the full affine group.

By Proposition~\ref{prop:full-invariance}, we know that \(\dA(f)=0\), \(\rA(f)=0\), constancy of \(\omega_f\), and the existence of an orientation-preserving real-affine $A$ for which \(A\circ f\) is analytic are equivalent for any orientation-preserving harmonic mapping \(f=h+\overline{g}\) on \(\mathbb{D}\).
The following result holds.

\begin{proposition}
\label{prop:HD-stable}
For every $D\ge0$, the classes $\mathscr H(D)$ and $\mathscr S(D)$ are invariant under the following classes of mappings:
\begin{enumerate}[label=\textup{(\roman*)}]
\item all orientation-preserving real-affine postcompositions;
\item all conformal automorphisms of the source disk;
\item all target similarities and the normalizations used in harmonic Koebe transforms.
\end{enumerate}
Furthermore,
\[
 \mathscr H(0)
 =\{\text{Target-affine images of locally univalent analytic mappings}\},
\]
and similarly in the univalent category.
\end{proposition}

\vskip.20in

\section{The sharp pre‑Schwarzian gap for the affine‑optimal analytic component}\label{sec:preschwarzian}

We apply the Schwarz--Pick lemma for the full dilatation disk. For
$F=H+\overline G$, with $\nu=G'/H'$, \eqref{eq:harmonic-P} gives
\begin{equation}\label{eq:P-gap}
 P_H-P_F=\frac{\overline\nu\,\nu'}{1-|\nu|^2}.
\end{equation}
The classical Schwarz–Pick lemma for $\nu$ yields
\[
(1-|z|^2)|P_H-P_F|\le |\nu(z)|\le k.
\]
For \(k>0\), an application of the Schwarz–Pick lemma to \(\nu/k\) employs the full dilatation constraint and yields the following result.

\begin{lemma}\label{lem:scaled-SP}
If $\nu:\D\to k\D$ is analytic and $0<k<1$, then
\begin{equation}\label{eq:scaled-SP}
 |\nu'(z)|\le\frac{k^2-|\nu(z)|^2}{k(1-|z|^2)},
\end{equation}
with equality at a point if and only if $\nu/k\in\Aut(\D)$.
\end{lemma}

\begin{figure}[htbp]
\centering
\begin{tikzpicture}[x=6.75cm,y=3.62cm,font=\small]
  \foreach \x in {0.25,0.50,0.75,1.00} \draw[black!11] (\x,0) -- (\x,1.02);
  \foreach \y in {0.25,0.50,0.75,1.00} \draw[black!11] (0,\y) -- (1,\y);
  \draw[thick,densely dashed] plot[domain=0:1,samples=51] (\x,{2*\x*\x/(3*sqrt(3))});
  \draw[very thick,blue!55!black] plot coordinates {
    (0.000000,0.0000000) (0.020000,0.0001540) (0.040000,0.0006162) (0.060000,0.0013873) (0.080000,0.0024686)
    (0.100000,0.0038619) (0.120000,0.0055694) (0.140000,0.0075939) (0.160000,0.0099387) (0.180000,0.0126079)
    (0.200000,0.0156060) (0.220000,0.0189381) (0.240000,0.0226101) (0.260000,0.0266286) (0.280000,0.0310010)
    (0.300000,0.0357355) (0.320000,0.0408412) (0.340000,0.0463283) (0.360000,0.0522078) (0.380000,0.0584923)
    (0.400000,0.0651953) (0.420000,0.0723322) (0.440000,0.0799196) (0.460000,0.0879762) (0.480000,0.0965227)
    (0.500000,0.1055822) (0.520000,0.1151803) (0.540000,0.1253459) (0.560000,0.1361113) (0.580000,0.1475128)
    (0.600000,0.1595918) (0.620000,0.1723951) (0.640000,0.1859766) (0.660000,0.2003978) (0.680000,0.2157305)
    (0.700000,0.2320585) (0.720000,0.2494803) (0.740000,0.2681136) (0.760000,0.2881003) (0.780000,0.3096137)
    (0.800000,0.3328692) (0.820000,0.3581400) (0.840000,0.3857800) (0.860000,0.4162621) (0.880000,0.4502408)
    (0.900000,0.4886660) (0.920000,0.5330083) (0.940000,0.5857657) (0.960000,0.6518552) (0.980000,0.7441076)
    (0.980400,0.7464372) (0.983100,0.7629539) (0.985600,0.7796964) (0.987900,0.7966688) (0.990000,0.8138757)
    (0.991900,0.8313219) (0.993600,0.8490124) (0.995100,0.8669522) (0.996400,0.8851469) (0.997500,0.9036020)
    (0.998400,0.9223234) (0.999100,0.9413173) (0.999600,0.9605902) (0.999900,0.9801488) (0.999990,0.9936904)
    (0.999999,0.9980015) (1.000000,1.0000000)
  };
  \draw[blue!55!black,fill=white] (1,1) circle[radius=1.45pt];
  \draw[->] (0,0) -- (1.055,0) node[right] {$k$};
  \draw[->] (0,0) -- (0,1.08) node[above] {$\mathcal C$};
  \node[below left] at (0,0) {$0$};
  \foreach \x/\labelx in {0.25/0.25,0.5/0.5,0.75/0.75,1/1}
    \draw (\x,0.022) -- (\x,-0.022) node[below] {$\labelx$};
  \foreach \y/\labely in {0.25/0.25,0.5/0.5,0.75/0.75,1/1}
    \draw (0.012,\y) -- (-0.012,\y) node[left] {$\labely$};
  \draw[very thick,blue!55!black] (0.055,0.92) -- (0.14,0.92)
    node[right,black] {$\mathcal C(k)$};
  \draw[thick,densely dashed] (0.055,0.77) -- (0.14,0.77)
    node[right] {$\frac{2}{3\sqrt3}k^2$};
\end{tikzpicture}
\caption{The sharp pre-Schwarzian gap $\mathcal C(k)$ in Theorem~\ref{thm:intro-ps}
  and its leading quadratic term at $k=0$. The open circle marks the limit at $k=1$.}
\label{fig:thm1_4_Ck_plot}
\end{figure}

We are now ready to prove Theorem~\ref{thm:intro-ps}.

\begin{proof}[Proof of Theorem~\ref{thm:intro-ps}]
First consider any orientation-preserving harmonic representative
$F=H+\overline G$ with $k=\norm{G'/H'}_\infty<1$. For $k=0$, the gap
vanishes by \eqref{eq:P-gap}. For $k>0$, if $\nu$ is constant its
derivative is zero, while if $\nu$ is nonconstant the maximum modulus
principle yields $\nu(\D)\subset k\D$. Thus \eqref{eq:P-gap} and
Lemma~\ref{lem:scaled-SP} give
\begin{equation}\label{eq:pointwise-C}
 (1-|z|^2)|P_H(z)-P_F(z)|
 \le \Phi_k(|\nu(z)|),
 \quad
 \Phi_k(t)=\frac{t(k^2-t^2)}{k(1-t^2)},\quad 0\le t\le k .
\end{equation}
The factor $k^2-t^2$ is the gain from scaling Schwarz--Pick lemma. To keep
that gain in the norm estimate, we maximize the entire expression over
$[0,k]$.

Direct differentiation gives
\[
 \Phi_k'(t)
 =\frac{t^4+(k^2-3)t^2+k^2}{k(1-t^2)^2}.
\]
The denominator is positive on $[0,k]$, so the sign is determined by the quadratic
\begin{equation}\label{eq:critical-poly}
 Q_k(x)=x^2+(k^2-3)x+k^2,\quad x=t^2.
\end{equation}
Here \[Q_k(0)=k^2>0\] and \[Q_k(k^2)=2k^2(k^2-1)<0.\] The quadratic
therefore has exactly one zero in $(0,k^2)$, namely, its smaller root
\begin{equation}\label{eq:xk}
 x_k=\frac{3-k^2-\sqrt{(1-k^2)(9-k^2)}}{2}.
\end{equation}
Consequently, the maximum is attained only at $t_k=\sqrt{x_k}$. The
identity
\begin{equation}\label{eq:k2-xrelation}
 k^2=\frac{x_k(3-x_k)}{1+x_k}
\end{equation}
implies \[k^2-x_k=2\frac{x_k(1-x_k)}{1+x_k},\] so
\begin{equation}\label{eq:Ck}
 \mathcal C(k)=\Phi_k(t_k)
 =\frac{\sqrt{x_k}(k^2-x_k)}{k(1-x_k)}
 =\frac{2x_k^{3/2}}{k(1+x_k)}.
\end{equation}
In particular, \eqref{eq:k2-xrelation} and $x_k<k^2$ give
\[x_k=\frac{k^2}{3}+O(k^4)\] as $k\to0$. Substitution in the last expression in
\eqref{eq:Ck} yields
\begin{equation}\label{eq:Ck-small}
 \mathcal C(k)=\frac{2}{3\sqrt3}k^2+O(k^4).
\end{equation}

At any point where the weighted gap equals $\mathcal C(k)>0$, equality
in \eqref{eq:pointwise-C} requires both $|\nu(z)|=t_k$ and equality in
Schwarz--Pick lemma for $\nu/k$. Conversely, if $\nu/k$ is a disk
automorphism, it takes some point to modulus $t_k/k$, and both
equalities hold there. For the mapping $f_k$ of \eqref{63}, the
dilatation is $\nu(z)=kz$. Writing $r=|z|$ and using $P_H=0$, we have
\[
 (1-r^2)|P_H-P_{f_k}|
 =\frac{k^2r(1-r^2)}{1-k^2r^2}
 =\Phi_k(kr).
\]
This attains the bound on $|z|=t_k/k$. Lemma~\ref{lem:realization} gives global univalence and
quasiconformality of $f_k$. Its dilatation image is the hyperbolic disk
$k\D$ centered at $0$, so $f_k$ is already balanced. In particular, the
universal estimate and its restriction to balanced mappings have the
same sharp constant.

Now take $F=\Bal[f]$ and let $H_A$ denote the analytic part of
$A\circ f$ for an orientation-preserving real-affine $A$. Target-affine
invariance gives \[P_F=P_{A\circ f}=P_f.\] By
Theorems~\ref{thm:intro-optimal} and~\ref{thm:canonical-section}, we have
\begin{equation}\label{eq:k-affine-min}
 \norm{G'/H'}_\infty
 =\kA(f)=\min_A\norm{\omega_{A\circ f}}_\infty.
\end{equation}
Indeed, $k\mapsto(1+k)/(1-k)$ takes the dilatation norm to maximal
differential distortion and is strictly increasing. Each affine
representative has dilatation norm $k_A<1$ since the induced disk
automorphism maps the compact closure of $\omega_f(\D)$ into $\D$.
For fixed $0<t<k$, the expression
\[\Phi_k(t)=\frac{t}{1-t^2}\left(k-\frac{t^2}{k}\right)\] is strictly increasing in $k$.
If $0<k_1<k_2<1$, testing $\Phi_{k_2}$ at $t_{k_1}$ proves
$\mathcal C(k_2)>\mathcal C(k_1)$. Together with $\mathcal C(0)=0$,
this proves that $\Bal[f]$ minimizes the universal upper bound
$\mathcal C(k_A)$ for $\norm{P_{H_A}-P_f}$ over the affine orbit.
\end{proof}

\begin{remark}\label{rem:quadratic-significance}
The quadratic behavior arises from scaled Schwarz--Pick lemma in every
fixed affine gauge, as shown in Figure~\ref{fig:thm1_4_Ck_plot}. Balancing selects the smallest parameter in this
sharp family of estimates.
\end{remark}

\vskip.20in
\section{Further geometric consequences}\label{sec:further}

We collect several consequences that clarify the geometry of the new invariants.

\subsection{The center-displacement estimate}

Let $f$ be an orientation-preserving harmonic mapping with $\norm{\omega_f}_\infty<1$, let $c=\cA(f)$, and set $R=\rA(f)$.  Define the ordinary hyperbolic radius from the origin
\begin{equation}\label{eq:r0-def}
 r_0(f)=\sup_{\zeta\in E_f}d_{\D}(0,\zeta)=\log K(f).
\end{equation}
Then
\begin{equation}\label{eq:r0-triangle}
 R\le r_0(f)\le R+d_{\D}(0,c).
\end{equation}
The first inequality is minimality of the circumradius, and the second is the triangle inequality.  If $E_f$ is a hyperbolic ball centered at $c$, equality holds in the second estimate.  Thus the excess ordinary distortion over the optimal affine distortion is controlled by the circumradius together with the displacement of the chosen conformal origin from the circumcenter; for hyperbolic-ball dilatation ranges this control is exact.

In exponential form, we have
\begin{equation}\label{eq:K-center-displacement}
 \KA(f)
 \le K(f)
 \le \KA(f)e^{d_{\D}(0,\cA(f))}.
\end{equation}
This makes the gauge dependence explicit.

\subsection{Support-type dichotomy}

The support theorem separates the smallest enclosing disk into two geometric types.  We formulate the classification for the compact set $E_f$.

As a direct consequence of the three‑point support principle (Theorem \ref{thm:intro-support}), the hyperbolic circumcenter of \(E_f\) is determined either by exactly two or by exactly three contact points on \(\partial E_f\). This dichotomy motivates our classification into two‑support type and three‑support type harmonic mappings given below.

\begin{definition}
An orientation-preserving harmonic mapping of uniformly bounded differential distortion is of \emph{two-support type} if the circumcenter of $E_f$ is determined by two contact points; it is of \emph{three-support type} if no pair of contact points determines the circumcenter.
\end{definition}

If $f$ is of two-support type, Theorem~\ref{thm:support} gives contact points $\xi_1,\xi_2$ with
\(
 d_{\D}(\xi_1,\xi_2)=2\rA(f),
\)
and therefore
\begin{equation}\label{eq:two-support-diam}
 \rA(f)=\frac{\dA(f)}2.
\end{equation}
Conversely, because $E_f$ is compact, its diameter is attained.  If $\dA(f)=2\rA(f)$ and $\xi,\eta\in E_f$ satisfy $d_{\D}(\xi,\eta)=\dA(f)$, then
\[
 2\rA(f)=d_{\D}(\xi,\eta)
 \le d_{\D}(\xi,\cA(f))+d_{\D}(\cA(f),\eta)
 \le2\rA(f).
\]
Equality throughout forces both points to be contact points and $\cA(f)$ to lie on the geodesic segment $[\xi,\eta]$; because the two endpoint distances are both $\rA(f)$, the center is their midpoint.  Hence
\begin{equation}\label{eq:lower-equality-classification}
 \rA(f)=\frac{\dA(f)}2
\Longleftrightarrow
 f\text{ is of two-support type}.
\end{equation}

Three-support type is the genuinely two-dimensional alternative.  The regular equilateral hyperbolic triangle realizes the largest possible radius for fixed diameter, but a general three-support configuration need not be equilateral.  Thus the interval
\[
 \frac D2\le R\le J(D)
\]
interpolates between a diametral two-point configuration and the equilateral extremal configuration.
\subsection{Convex-hull reduction}

By Lemma~\ref{lem:rad-basic}, we have
\begin{equation}\label{eq:convex-hull-invariants}
 \rA(f)
 =\rad_{\D}(\operatorname{hconv}E_f),
 \quad
 \dA(f)
 =\diam_{\D}(\operatorname{hconv}E_f).
\end{equation}
Thus all affine distortion information developed here depends only on the hyperbolic convex hull of the dilatation range.  Fine internal topology of $\omega_f(\D)$ is invisible to the first-order optimization problem.  This suggests that additional affine invariants should be sought among hyperbolic convex-geometric functionals of $E_f$, for example inradius, width, area of the convex hull, or capacities.  The work of Nasser--Rainio--Vuorinen~\cite{NasserRainioVuorinen2022} shows that hyperbolic diameter and Jung radius interact nontrivially with conformal capacity, making this a potentially useful direction.

\vskip.20in

\section{Open problems}\label{sec:open}

The affine‑balancing theory developed in this paper raises several natural open problems stemming from our main results. We collect selected ones below.

\begin{problem}
\label{prob:higher-moments} Develop affine‑invariant distortion theory based on other hyperbolic convex‑geometric functionals of $E_f$, such as width, inradius, area of $\operatorname{hconv} E_f$, or conformal capacity. Determine which of these control geometric properties of $f(\mathbb{D})$ more efficiently than $r_{\mathrm{aff}}$ and $d_{\mathrm{aff}}$ alone.\end{problem}

For a given affine orbit, $K_{\mathrm{aff}}(f)$ and $d_{\mathrm{aff}}(f)$ furnish intrinsic, coordinate‑free measures of global dilatation, independent of artificial target‑affine normalizations. Since the canonical balanced representative $\mathscr{B}[f]$ realizes the minimal possible dilatation parameter within its orbit, one may hope that boundary regularity properties of its image domain can be quantitatively bounded purely in terms of these orbit‑invariant quantities.

\begin{problem}
\label{prob:boundary} Let $F=\mathscr{B}[f]$ be the canonical representative of a harmonic quasiconformal homeomorphism, with $F(\mathbb{D})$ a smooth or chord‑arc Jordan domain. Determine which boundary Lipschitz, co‑Lipschitz, or Hölder estimates admit quantitative dependence on $K_{\mathrm{aff}}(f)$ or $d_{\mathrm{aff}}(f)$ and the geometry of $F(\mathbb{D})$.\end{problem}
For boundary regularity of harmonic quasiconformal mappings; see Kalaj~\cite{Kalaj2008,Kalaj2012,Kalaj2015,Kalaj2022}.

\begin{problem}
\label{prob:higher-dim} Find an analogue of affine balancing for harmonic quasiconformal mappings in $\mathbb{R}^n$, $n\geq 3$. In the plane, the space of linear conformal structures is the hyperbolic disk and the complex dilatation provides a holomorphic coordinate. In higher dimensions, the space of oriented linear conformal structures, after quotienting by scale, may be identified with the symmetric space $\mathrm{SL}(n,\mathbb{R})/\mathrm{SO}(n)$, and the local distortion tensor replaces $\omega$. Is the optimal target‑affine distortion determined by a circumcenter problem in this nonpositively curved symmetric space, and under what hypotheses does this recover the quasiconformal constant of a harmonic homeomorphism?\end{problem}

\section*{Concluding remarks}\label{sec:conclusion}

The maximal differential distortion of a harmonic mapping depends on the affine gauge because a target-affine mapping acts nontrivially on the complex dilatation; for a univalent harmonic mapping, this is the quasiconformal constant.  Instead of treating this as an obstruction to affine invariance, one can use the action itself: it is exactly the isometry group action of $\Aut(\D)$ on the Poincar\'e disk.  The smallest possible maximal differential distortion in an affine orbit is therefore a hyperbolic smallest-enclosing-ball problem.

This viewpoint produces several concrete consequences.  The affine circumradius is the logarithm of the optimal maximal differential distortion (hence of the optimal quasiconformal constant in the univalent category); the minimizer is unique modulo similarities; two or three support points in the closure of the dilatation range determine it; the sharp hyperbolic Jung theorem converts pairwise dilatation oscillation into a sharp affine distortion estimate; and the zero-diameter locus is exactly the class of affine images of analytic mappings.  Passing to the balanced representative gives a canonical section of the affine quotient and turns fixed affine diameter into a genuinely affine-linear invariant bounded-distortion constraint.

The resulting framework places classical analytic linear invariant families at the zero-defect endpoint and harmonic bounded-distortion families at positive affine diameter.  It also suggests a broader principle: when a geometric distortion parameter transforms by isometries of a nonpositively curved parameter space, the correct invariant normalization is often obtained by taking the circumcenter of its full range rather than normalizing it at one base point.

\vskip .05in
\noindent{\bf Acknowledgements.}
Z.-G. Wang was partially supported by the \textit{Key Project of Education Department of Hunan Province} under Grant no. 25A0668, and
the \textit{Natural Science Foundation of Changsha} under Grant no. kq2502003
of the P. R. China. D. Zhong was partially supported by the \textit{Guangdong Basic and Applied Basic Research Foundation} under Grant nos. 2022A1515110967 and 2023A1515011809 of the P. R. China. 

\vskip .05in
\noindent{\bf
Contribution statement.}
All authors contributed equally to this work.

\vskip .05in
\noindent{\bf Conflicts of interest.} The authors declare that they have no conflict of interest.

\vskip .05in
\noindent{\bf Data availability statement.}  Data sharing is not applicable to this article as no datasets were generated or analysed during the current study.


\end{document}